\documentclass[12pt,leqno]{amsart}
\usepackage{amsmath}
\usepackage{amssymb}
\usepackage{calrsfs}
\newcommand{\N}{{\mathbb{N}}}

\newcommand{\cS}{{\mathcal{S}}}
\newcommand{\cW}{{\mathcal{W}}}
\renewcommand{\leq}{\leqslant}
\renewcommand{\geq}{\geqslant}

\newtheorem{thm}{Theorem}
\newtheorem{cor}[thm]{Corollary}
\newtheorem{prop}[thm]{Proposition}
\newtheorem{lem}[thm]{Lemma}

\theoremstyle{definition}

\theoremstyle{remark}
\newtheorem{rem}[thm]{Remark}

\begin{document}
\vspace*{-1.5cm}

\title{Coxeter groups and automorphisms}
\author{Meinolf Geck and Lacrimioara Iancu}
\address{IAZ -- Lehrstuhl f\"ur Algebra, Universit\"at Stuttgart, 
Pfaffenwaldring 57, D--70569 Stuttgart, Germany} 
\email{meinolf.geck@mathematik.uni-stuttgart.de}
\email{iancu@mathematik.uni-stuttgart.de}

\begin{abstract} Let $(W,S)$ be a Coxeter system and $\Gamma$
be a group of automorphisms of $W$ such that $\gamma(S)=S$ for all
$\gamma \in \Gamma$. Then it is known that the group of
fixed points $W^\Gamma$ is again a Coxeter group with a canonically
defined set of generators. The usual proofs of this fact rely
on the reflection representation of $W$. Here, we give a proof which 
only uses the combinatorics of reduced expressions in $W$. As a by-product, 
this shows that the length function on $W$
restricts to a weight function on $W^\Gamma$.
\end{abstract}

\maketitle

\date{}

\pagestyle{myheadings}
\markboth{Geck and Iancu}{Coxeter groups and automorphisms}

Let $(W,S)$ be a Coxeter system where $S$ is finite. Let $l\colon W
\rightarrow \N_0$ be the corresponding length function. Let $\Gamma$ be
a group of automorphisms of $W$ such that $\gamma(S)=S$ for all $\gamma
\in \Gamma$. Then we have $l(\gamma(w))=l(w)$ for all $w\in W$. Let 
\[W^\Gamma:= \{w\in W\mid \gamma(w)=w\mbox{ for all $\gamma\in \Gamma$}\}\]
be the group of fixed points. For any subset $I\subseteq S$ let 
$W_I\subseteq W$ be the corresponding parabolic subgroup. If $W_I$ is 
finite, let $w_I\in W_I$ be the longest element. Let $\bar{S}$ be the set 
of $\Gamma$-orbits $I$ on $S$ such that $W_I$ is finite; note that 
$w_I\in W^\Gamma$ for $I\in\bar{S}$. The purpose of this note is to 
provide a short proof of the following (known) result. 

\begin{thm} \label{thm1} The pair $(W^\Gamma,\{w_I\mid I\in \bar{S}\})$ is 
a Coxeter system. Let $l_\Gamma\colon W^\Gamma\rightarrow \N_0$ be the 
corresponding length function. Then, for any $w,w'\in W^\Gamma$, we have 
$l(ww')=l(w)+ l(w')$ if and only if $l_\Gamma(ww')=l_\Gamma(w)+l_\Gamma(w')$.
\end{thm}

If $W$ is finite, this is due to Steinberg \cite[\S 11]{Stein}; for
general $(W,S)$, see H\'ee \cite{hee} or Lusztig \cite[Appendix]{Lusztig03}.
The proofs in [{\em loc.\ cit.}] rely on properties of the reflection
representation of $W$. The proof that we shall give here is based on notes 
of a course on Coxeter groups given by the second-named author at EPFL 
in 2004. It proceeds somewhat more directly by using only the 
combinatorics of reduced expressions of elements in $W$. 

A key role is played by dihedral groups and distinguished coset 
representatives with respect to a parabolic subgroup of $W$ (see, for 
example, \cite[\S 2.1]{gepf}). Also recall that a parabolic subgroup $W_I$ 
is finite if and only if there exists an element $u\in W_I$ such that 
$l(su)<l(u)$ for all $s\in I$, in which case we have $u=w_I$; note also that
$w_I^2=1$. (For these facts see, for example, \cite[\S 1.5]{gepf}.)

\begin{lem}[Cf.\ \protect{\cite[3.4]{hee}, \cite[A.1(a)]{Lusztig03}}] 
\label{lem1} Let $w\in W^\Gamma$. Then we can write $w=w_{J_1}\cdots w_{J_r}$
where $J_i\in \bar{S}$ and $l(w)=l(w_{J_1})+ \ldots +l(w_{J_r})$. Furthermore, 
if $s\in S$ is such that $l(sw)<l(w)$, then we can choose $J_1$ such 
that $s\in J_1$. In particular, $W^\Gamma=\langle w_I \mid I\in 
\bar{S}\rangle$.
\end{lem}

\begin{proof} Induction on $l(w)$. If $l(w)=0$, then $w=1$ and there is 
nothing to prove. Now let $l(w)>0$ and $s\in S$ be such that $l(sw)<l(w)$.
Let $J_1$ be the $\Gamma$-orbit of~$s$ ; since $l(sw)<l(w)$, we
have $l(\gamma(s)w)=l(\gamma(sw))<l(w)$ for all $\gamma\in \Gamma$
and, hence, $l(tw)<l(w)$ for all $t\in J_1$.
Let now $X_{J_1}=\{x\in W\mid l(tx)>l(x) 
\mbox{ for all $t\in J_1$}\}$ be the set of distinguished right coset 
representatives of $W_{J_1}$ in $W$. We can write $w=ux$ where $u\in 
W_{J_1}$, $x\in X_{J_1}$ and $l(w)=l(u)+l(x)$. 
Since $tu\in W_{J_1}$, 
we have $l(tux)=l(tu)+l(x)$ for all $t\in J_1$. So we conclude that 
$l(tu)<l(u)$ for all $t\in J_1$. Hence, $W_{J_1}$ must be finite and 
$u=w_{J_1}\in W^\Gamma$. But then we also have $x\in W^\Gamma$ and we 
can continue with $x$ by induction. 
\end{proof}

In what follows, to simplify notation, we shall write $w=x\bullet y$
if $w,x,y\in W$ are such that $w=xy$ and $l(w)=l(x)+l(y)$.  Thus, in
the setting of Lemma~\ref{lem1}, we can write $w=w_{J_1}\bullet \ldots 
\bullet w_{J_r}$.

\begin{rem} \label{rem1} Let $I,J\in \bar{S}$ and assume that $I\neq J$.
Let $K:=I\cup J$. Applying Lemma~\ref{lem1} to $W_K$ shows that 
$W_K^\Gamma=\langle w_I,w_J\rangle$ is a dihedral group. Suppose that there 
exists some $u\in W_K$ such that $l(su)<l(u)$ for all $s\in K$. Then 
$W_K$ is finite and $u=w_K\in W_K^\Gamma$. Being a dihedral group, the 
order of $W_K^\Gamma$ is $2m$ for some $m\in \N$. The elements of 
$W_K^\Gamma$ are products of the form $w_Iw_Jw_I \cdots$ or 
$w_Jw_Iw_J \cdots$, with at most $m$ factors; furthermore, two such
products (one starting with $w_I$ and one starting with $w_J$) are equal 
if and only if there are exactly $m$ factors in each of them. This also 
shows that $l(y)\leq \frac{m}{2}(l(w_I)+l(w_J))$ for all $y\in W_K^\Gamma$. 
We now claim that 
\[w_K=\underbrace{w_I\bullet w_J\bullet w_I\bullet \ldots}_{\text{$m$ terms}}
= \underbrace{w_J\bullet w_I\bullet w_J \bullet \ldots}_{\text{$m$ terms}} 
\quad \mbox{where} \quad l(w_K)=\frac{m}{2}(l(w_I)+ l(w_J)).\]
Indeed, by Lemma~\ref{lem1}, we can write $w_K=w_I\bullet w_J \bullet 
w_I\bullet \ldots$ with, say $p\geq 1$, terms, and also $w_K=w_J\bullet w_I 
\bullet w_J\bullet \ldots$ with, say $q\geq 1$, terms. Since $l(w_K)\leq 
\frac{m}{2}(l(w_I)+ l(w_J))$, we must have $p\leq m$ and $q\leq m$. But 
then the two products $w_K=w_I\bullet w_J\bullet w_I\bullet\ldots=w_J
\bullet w_I\bullet w_J\bullet \ldots$ can only be equal if there are 
exactly $m$ factors in both sides. Thus, we have $p=q=m$, as required.  
\end{rem}

\begin{lem} \label{lem2} Let $w\in W^\Gamma$ and assume that we have two
expressions 
\[w=w_{J_1}\bullet \ldots \bullet w_{J_r}=w_{I_1}\bullet \ldots \bullet 
w_{I_p}\]
where $J_i,I_i\in \bar{S}$. Then $r=p$.
\end{lem}

\begin{proof} Induction on $l(w)$. If $l(w)=0$, then $w=1$ and there is
nothing to prove. Now assume that $l(w)>0$; then $r\geq 1$ and $p\geq 1$. 
If $I_1=J_1$, then $w':=w_{I_1}w=w_{J_2}\bullet \ldots \bullet w_{J_r}=
w_{I_2}\bullet \ldots \bullet w_{I_p}$. So $r-1=p-1$ by induction. Now 
assume that $I_1\neq J_1$. Let $K:=I_1 \cup J_1$ and $X_K=\{x\in W\mid 
l(sx)>l(x) \mbox{ for all $s\in K$}\}$ be the set of distinguished coset 
representatives of $W_K$ in $W$. We can write $w=u\bullet x$ where 
$u\in W_K$ and $x\in X_K$. We have $l(sw)<l(w)$ for all $s\in K$ and so 
$l(su)<l(u)$ for all $s\in K$. Hence, by Remark~\ref{rem1}, $W_K$ must be 
finite and $u=w_K\in W_K^\Gamma$. Then we also have $x\in W^\Gamma$ and so,
by Lemma~\ref{lem1}, we can write $x=w_{L_1}\bullet \ldots \bullet 
w_{L_q}$ where $L_i\in \bar{S}$. Now consider the identities:
\[w_{J_1}\bullet \ldots \bullet w_{J_r}=w=w_K\bullet x=\underbrace{(w_{J_1}
\bullet w_{I_1}\bullet w_{J_1}\bullet \ldots )}_{\text{$m$ terms}} 
\bullet w_{L_1}\bullet \ldots \bullet w_{L_q}.\]
Cancelling $w_{J_1}$ on the left on both sides, we deduce that 
\[w_{J_2}\bullet \ldots \bullet w_{J_r}=\underbrace{(w_{I_1}\bullet w_{J_1}
\bullet \ldots)}_{\text{$m-1$ terms}} \bullet w_{L_1}\bullet \ldots 
\bullet w_{L_q}.\]
By induction, we conclude that $r-1=(m-1)+q$. Applying the same argument 
to 
\[w_{I_1}\bullet \ldots \bullet w_{I_p}=w=w_K\bullet x=\underbrace{(w_{I_1}
\bullet w_{J_1}\bullet w_{I_1}\bullet \ldots )}_{\text{$m$ terms}} 
\bullet w_{L_1}\bullet \ldots \bullet w_{L_q}.\]
also yields $p-1=(m-1)+q$. Consequently, we obtain $r=p$, as desired. 
\end{proof}

The above proof is inspired by the proof of the Matsumoto--Tits Lemma
in \cite[1.2.2]{gepf}, which in turn follows Tits \cite[Cor.~II.1.12]{Tits1}.

Now let $\lambda\colon W^\Gamma\rightarrow \N_0$ denote the length 
function with respect to the generators $\{w_I\mid I\in\bar{S}\}$. Thus,
the properties in \cite[Chap.~IV, n$^\text{o}$~1.1]{bour} do hold
for $\lambda$ but note, for example, that it is not yet clear that 
$\lambda(w_Iw)\neq \lambda(w)$ for $w\in W^\Gamma$ and $I\in\bar{S}$. 

\begin{lem} \label{lem1a} Let $w\in W^\Gamma$ and $w=w_{J_1}\cdots w_{J_p}$
where $J_i\in \bar{S}$. If $p=\lambda(w)$, then $w=w_{J_1}\bullet \ldots
\bullet w_{J_p}$. 
\end{lem}

\begin{proof} Induction on $p$. If $p=0$ or $1$, then the assertion is clear. 
Now assume that $p\geq 2$ and set $w':=w_{J_2}\cdots w_{J_p}\in W^\Gamma$.
Then $\lambda(w')=p-1$ and so, by induction, $w'=w_{J_2}\bullet \ldots 
\bullet w_{J_r}$. Now we distinguish two cases. If $l(sw')>l(w')$ for 
all $s\in J_1$, then $w'\in X_{J_1}$ and so $l(w_{J_1}w')=l(w_{J_1})+l(w)$. 
Thus, $w=w_{J_1}\bullet w'$ and the desired assertion is proved. On
the other hand, if $l(sw')<l(w')$ for some $s\in J_1$, then we can also 
find an expression $w'=w_{L_1}\bullet \ldots \bullet w_{L_q}$ where 
$L_i\in \bar{S}$ and $L_1=J_1$; see Lemma~\ref{lem1}. By Lemma~\ref{lem2}, 
we have $p-1=q$. Now $w=w_{J_1}w'=w_{L_2}\cdots w_{L_{p-1}}$ and so 
$\lambda(w)<p$, a contradiction. Hence, this case does not occur.
\end{proof}

%
%
\begin{cor} \label{cor1} Let $w,w'\in W^\Gamma$. Then $l(ww')=l(w)+
l(w')$ if and only if $\lambda(ww')=\lambda(w)+\lambda(w')$.
In particular, the restriction of $l$ to $W^\Gamma$ is a weight
function in the sense of Lusztig \cite{Lusztig03}.
\end{cor}

\begin{proof} Let $r=\lambda(w)$ and $p=\lambda(w')$. By Lemma~\ref{lem1a},
we have $w=w_{J_1}\bullet \ldots \bullet w_{J_r}$ and 
$w'=w_{I_1}\bullet \ldots \bullet w_{I_p}$ where $J_i,I_i\in\bar{S}$. 

Assume first that $l(ww')=l(w)+l(w')$. Then $ww'=w_{J_1} \bullet \ldots 
\bullet w_{J_r}\bullet w_{I_1}\bullet \ldots \bullet w_{I_p}$. Let 
$q:=\lambda(ww') \leq r+p$. Again, by Lemma~\ref{lem1a}, we have 
$ww'=w_{L_1}\bullet \ldots \bullet w_{L_q}$ where $L_i\in \bar{S}$. 
Now Lemma~\ref{lem2} implies that $q=r+p$, as desired. Conversely, assume 
that $\lambda(ww')=\lambda(w)+\lambda(w')$. Since 
we have $ww'=w_{J_1} \cdots w_{J_r}w_{I_1}\cdots w_{I_p}$, Lemma~\ref{lem1a} 
shows once more that $ww'=w_{J_1} \bullet \ldots \bullet w_{J_r}\bullet 
w_{I_1} \bullet \ldots \bullet w_{I_p}$. Thus, we have $l(ww')=l(w)+l(w')$.
\end{proof}

Since $(W,S)$ is a Coxeter system, the ''Exchange Condition'' holds. Recall
that this means the following. Let $w\in W$ and $s\in S$. Let $p=l(w)$
and $w=s_1\cdots s_p$ where $s_i\in S$. If $l(sw)\leq l(w)$, then there 
exists some $i\in \{1,\ldots,p\}$ such that $sw=s_1\cdots s_{i-1}s_{i+1}
\cdots s_p$. We can now show that the pair $(W^\Gamma,\{w_I\mid I\in 
\bar{S}\})$ also satisfies this ''Exchange Condition'' and, hence,
$(W^\Gamma,\{w_I\mid I\in\bar{S}\})$ is a Coxeter system; see Bourbaki 
\cite[Chap.~IV, n$^\text{o}$~1.6]{bour}. In combination with 
Corollary~\ref{cor1}, this will complete the proof of Theorem~\ref{thm1}.

\begin{prop} \label{prop1} Let $w\in W^\Gamma$ and $I\in\bar{S}$. Let
$p=\lambda(w)$ and $w=w_{J_1}\cdots w_{J_p}$ where $J_i\in\bar{S}$.
If $\lambda(w_Iw)\leq \lambda(w)$, then there exists some $i\in
\{1,\ldots,p\}$ such that $w_Iw=w_{J_1}\cdots w_{J_{i-1}}w_{J_{i+1}}
\cdots w_{J_p}$.
\end{prop}

\begin{proof} If we had $l(sw)>l(w)$ for all $s\in I$, then $w\in X_I$ 
and so $l(w_Iw)=l(w_I)+l(w)$. Hence, Corollary~\ref{cor1} would imply
that $\lambda(w_Iw)=\lambda(w_I)+\lambda(w)>\lambda(w)$, contrary to 
our assumption. Thus, there exists some $s\in I$ such that $l(sw)\leq l(w)$. 
Further note that, by Lemma~\ref{lem1a}, we have $w=w_{J_1}\bullet \ldots 
\bullet w_{J_p}$. Taking reduced expressions for all $w_{J_i}$, we obtain
a reduced expression for $w$. Since the ''Exchange Condition'' holds for 
$(W,S)$, there exists an index $i\in\{1,\ldots,p\}$ such that 
\[ sw=w_{J_1}\cdots w_{J_{i-1}}xw_{J_{i+1}}\cdots w_{J_p}\]
where $x\in W_{J_i}$ is obtained by dropping one factor in a reduced 
expression for $w_{J_i}$. Consequently, we have 
\[z^{-1}sz=xw_{J_i}\qquad\mbox{where}\qquad z:=w_{J_1}\cdots w_{J_{i-1}}.\]
Since $z\in W^\Gamma$, we obtain $z^{-1}\gamma(s)z=\gamma(xw_{J_i}) \in 
W_{J_i}$ for all $\gamma\in \Gamma$ and so $z^{-1}w_Iz\in W_{J_i}$. Since 
also $z^{-1}w_Iz\in W^\Gamma$ and $W_{J_i}^\Gamma=\{1,w_{J_i}\}$ (see
Lemma~\ref{lem1}), we conclude that $z^{-1}w_Iz=w_{J_i}$. This yields 
$w_Iw_{J_1}\cdots w_{J_{i-1}}= w_{J_1}\cdots w_{J_{i-1}}w_{J_i}$ and so 
$w_Iw=w_{J_1} \cdots w_{J_{i-1}}w_{J_{i+1}} \cdots w_{J_p}$, as required.
\end{proof}

\begin{rem} \label{rem2} Using similar arguments, one can extend 
Theorem~\ref{thm1} to the following ''relative'' setting (see 
\cite[\S 5]{L76a} and \cite[Chap.~25]{Lusztig03}). We fix a subset 
$I_0\subseteq S$ such that $W_{I_0}$ is finite and $\gamma(I_0)=I_0$ 
for all $\gamma\in \Gamma$;  furthermore, we assume that 
\[w_{I_0\cup J}\in N_W(W_{I_0}) \qquad \mbox{for all $J\in \cS$},\]
where $\cS$ denotes the set of all $\Gamma$-orbits on $S\setminus I_0$ such
that $W_{I_0\cup J}$ is finite. Let 
\[\cW:=\{w\in X_{I_0} \mid wW_{I_0}=W_{I_0}w\}.\]
Then $\cW$ is a subgroup of $W$ such that $\gamma(\cW)=\cW$ for all
$\gamma\in\Gamma$. Hence, we can consider the group of fixed points
$\cW^\Gamma$. We set 
\[s_J:=w_{I_0\cup J}w_{I_0}=w_{I_0}w_{I_0 \cup J} \qquad \mbox{for 
each $J\in\cS$}.\]
Then one can show that $s_J\in \cW^\Gamma$ and $s_J^2=1$; furthermore, 
$(\cW^\Gamma, \{s_J\mid J\in \cS\})$ is a Coxeter system. 
(Theorem~\ref{thm1} is the special case where $I_0=\varnothing$.) We 
omit further details.
\end{rem}

 
\end{document}